\documentclass[12pt]{amsart}
\usepackage{amsfonts}
\usepackage{amssymb}
\input amssym.def
\usepackage{pstricks}
\usepackage[dvips]{hyperref}
\newtheorem{theorem}{Theorem}
\newtheorem{lemma}[theorem]{Lemma}
\newtheorem{prop}[theorem]{Proposition}
\newtheorem{cor}[theorem]{Corollary}
\theoremstyle{definition}
\newtheorem{rem}[theorem]{Remark}

\newcommand{\M}{\mathcal{M}}

\newcommand{\Z}{\mathbb{Z}}

\newcommand{\bdr}[1]{\partial\! #1}
\newcommand{\lr}[1]{\left<#1\right>}

\numberwithin{equation}{section}
\numberwithin{theorem}{section}
\author{B\l{}a\.zej Szepietowski}
\title[Generating level 2 mapping class group]{A finite generating set for the level 2 mapping class group of a nonorientable surface}
\address[]{Institute of Mathematics, Gda\'nsk University, Wita Stwosza 57,
80-952 Gda\'nsk, Poland} 

\email{blaszep@mat.ug.edu.pl}
\thanks{Supported by the MNiSW grant N N201 366436}

\begin{document}
\begin{abstract}
We obtain a finite set of generators for the level 2 mapping class group of a closed nonorientable surface of genus $g\ge 3$. This set consists of isotopy classes of Lickorish's Y-homeomorphisms also called crosscap slides.
\end{abstract}

\maketitle

\section{Introduction}
For a compact surface $F$ the {\it mapping class group} $\M(F)$ is the group of isotopy classes of all, orientation
preserving if $F$ is orientable, homeomorphisms $h\colon F\to F$. 
The {\it level two mapping class group} is the subgroup $\Gamma_2(F)$ of $\M(F)$ consisting of the isotopy classes of homeomorphisms inducing the identity on $H_1(F,\Z_2)$. Note that it is a normal subgroup of $\M(F)$ of finite index.
For a closed orientable surface $F$ of genus $g>1$ Humphries proved in \cite{Hum2} that $\Gamma_2(F)$ is equal to the normal closure in $\M(F)$ of the square of a Dehn twist about a nonseparating simple closed curve. Sato computed in \cite{Sato} the abelianization of $\Gamma_2(F)$ for orientable $F$ with at most one boundary component and genus $g\ge 3$.

In this paper we are interested in the case of a closed nonorientable surface, which will be denoted as $N$ or $N_g$, where $g$ is the genus.  The genus of a nonorientable surface is the number of projective planes in the connected sum decomposition. Lickorish defined in \cite{Lick1} a homeomorphism of $N$ that he called Y-homeomorphism and proved in \cite{Lick1,Lick2} that  $\M(N_g)$ is generated by Dehn twists and one isotopy class of Y-homeomorphisms for $g\ge 2$. Y-homeomorphisms were called crosscap slides in \cite{Kork} and also in this paper we use this name. Chillingworth found in \cite{Chill} a finite set of generators for $\M(N_g)$. 
The actions of $\M(N_g)$ on $H_1(N_g,\Z)$ and $H_1(N_g,\Z_2)$ we studied by McCarthy and Pinkall \cite{McCP} and by Gadgil and Pancholi \cite{GP}. They proved that all automorphisms of $H_1(N_g,\Z)$ preserving the $\Z_2$-valued intersection form are induced by homeomorphisms. Recently S. Horose \cite{Hir} proved that a homeomorphsim of a closed non-orientable surface $N$ standardly  embedded in the 4-sphere is extendable to a homeomorphism of the 4-sphere if and only if it preserves the Guillou-Marin quadratic form  on $H_1(N,\Z_2)$.
In \cite{Szep2} we proved that for $g\ge 2$ the group $\Gamma_2(N_g)$
is equal to the normal closure in $\M(N_g)$ of one crosscap slide. In this paper 
we obtain a finite collection of crosscap slides generating $\Gamma_2(N_g)$ for $g\ge 3$. We also prove that $\Gamma_2(N_3)$ is isomorphic to the level 2 congruence subgroup of $GL(2,\Z)$ and compute the abelianization of $\Gamma_2(N_4)$.

\medskip

\noindent{\bf Acknowledgment.} I am grateful to Professor Susumu Hirose for sending me his preprint \cite{Hir}.

\section{Preliminaries}
By a {\it simple closed curve} in $N$ we mean an embedding
$\gamma\colon S^1\to N$. Note that $\gamma$ has an orientation; the
curve with the opposite orientation but same image will be denoted by
$\gamma^{-1}$. 
By abuse of notation, we will usually identify a simple closed curve with its
oriented image and also with its isotopy class. 
According to whether a regular neighborhood of $\gamma$ is an annulus or a M\"obius strip, we call $\gamma$ respectively {\it two-} or {\it one-sided}. 

Given a two-sided simple closed curve $\gamma$, $T_\gamma$ denotes a Dehn
twist about $\gamma$. On a  nonorientable surface it is
impossible to distinguish between right and left twists, so the
direction of a twist $T_\gamma$ has to be specified for each curve
$\gamma$. In this paper it is usually indicated by arrows in a figure.
 Equivalently we may choose an orientation of a regular
neighborhood of $\gamma$. Then $T_\gamma$ denotes the right Dehn twist with
respect to the chosen orientation. Unless we specify which of the
two twists we mean, $T_\gamma$ denotes any of the
two possible twists. 
By abuse of notation we will use the same symbol to denote a
homeomorphism and its isotopy class.

\begin{figure}
\input{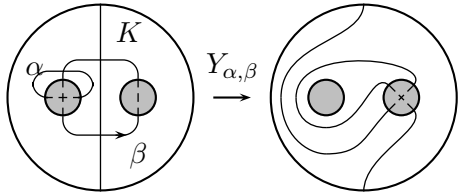}
\caption{\label{Y} Crosscap slide.}
\end{figure}

Suppose that $\alpha$ and $\beta$ are two simple closed curves in $N$,
such that $\alpha$ is one-sided, $\beta$ is two-sided and they
intersect in one point. Let $K\subset N$ be a regular neighborhood of
$\alpha\cup\beta$, 
which is homeomorphic to the Klein bottle with a
hole, and let $M\subset K$ be a regular neighborhood of $\alpha$, which
is a M\"obius strip. We denote by $Y_{\alpha,\beta}$ the {\it crosscap slide}, or {\it
Y-homeomorphism} equal to the identity on $N\backslash K$ and which may be
described as the result of pushing $M$ once along $\beta$ keeping the
boundary of $K$ fixed. Figure \ref{Y} illustrates the effect of
$Y_{\alpha,\beta}$ on an arc connecting two points in the boundary of
$K$. Here, and also in other figures of this paper, the shaded discs
represent crosscaps; this means that their interiors should be
removed, and then antipodal points in each resulting boundary
component should be identified.   
Up to isotopy, $Y_{\alpha,\beta}$ does not depend on the
choice of the regular neighbourhood $K$. 
The following
properties of crosscap slides are easy to verify.
\begin{equation}\label{y^-11}
Y_{\alpha^{-1},\beta}=Y_{\alpha,\beta},
\end{equation}
\begin{equation}\label{y^-1}
Y_{\alpha,\beta^{-1}}=Y^{-1}_{\alpha,\beta},
\end{equation}
\begin{equation}\label{hyh}
hY_{\alpha,\beta}h^{-1}=Y_{h(\alpha),h(\beta)},
\end{equation}
for all $h\in\M(N)$. 

Every crosscap slide induces an identity on $H_1(N,\Z_2)$, hence it belongs to $\Gamma_2(N)$.

\begin{theorem}[\cite{Szep2}]
For $g\ge 2$ the level 2 mapping class group $\Gamma_2(N_g)$ is generated by crosscap slides.
\end{theorem}

Fix $x_0\in N_g$ and define $\M(N_g,x_0)$ to be the group of isotopy classes all homeomorphisms of $h\colon N_g\to N_g$ such that $h(x_0)=x_0$. 
Let $U=\{z\in\mathbb{C}\,|\,|z|\le 1\}$ and fix an embedding
$e\colon U\to N_g$ such that $e(0)=x_0$.
The surface $N_{g+1}$ may be obtained by removing from $N_g$ the interior of $e(U)$ and then identifying $e(z)$ with $e(-z)$ for $z\in S^1=\bdr{U}$.
We define a {\it blowup homomorphism} 
\[\varphi\colon\M(N_g,x_0)\to\M(N_{g+1})\] as follows. 
Represent $h\in\M(N_g,x_0)$ by a homeomorphism $h\colon N_g\to N_g$ such that $h$ is equal to the identity on $e(U)$ or $h(x)=e(\overline{e^{-1}(x)})$ for $x\in e(U)$. Such $h$ commutes with the identification leading to $N_{g+1}$ and thus induces an element $\varphi(h)\in\M(N_{g+1})$. We refer the reader to \cite{Szep2} for a proof that $\varphi$ is well defined.

Forgetting the distinguished point $x_0$ induces a homomorphism
\[\M(N_g,x_0)\to\M(N_{g}),\] which fits into the Birman exact sequence (see \cite{Kork})
\[\pi_1(N_{g},x_0)\stackrel{j}{\to}\M(N_{g},x_0)\to\M(N_{g})\to 1.\]
The homomorphism $j$ is called {\it point pushing map}.
If $\gamma$ is a loop in $N_{g}$ based at $x_0$ and $[\gamma]\in\pi_1(N_{g},x_0)$ is  its homotopy class, then $j([\gamma])$ may be described as the effect of pushing $x_0$ once along $\gamma$. In order for $j$ to be a homomorphism, a product $[\gamma]\cdot[\delta]$ in $\pi_1(N_{g},x_0)$ means go along $\delta$ first and then along $\gamma$.

We define a {\it crosscap pushing map} to be the composition 
\[\psi=\varphi\circ j\colon\pi_1(N_{g},x_0)\to\M(N_{g+1}).\] 
Let $\alpha$ be the image in $N_{g+1}$ of $e(\bdr{U})$. Then $\alpha$ is a one-sided simple closed curve. Every simple loop $\gamma$ based at $x_0$ is homotopic to a loop $\gamma'$ which intersects $e(U)$ in two antipodal points. If $\beta$ is the image in $N_{g+1}$ of $\gamma'\backslash int(e(U))$ than $\beta$ is a simple closed curve, which intersects $\alpha$ in one-point, and which is two-sided if and only if $\gamma$ is one-sided.
The following lemma follows from the description of the point pushing map for nonorientable surfaces \cite[Lemma 2.2 and Lemma 2.3]{Kork} and the definition of a crosscap slide.

\begin{lemma}\label{push}
Suppose that $\gamma$ is a simple  loop in $N_{g}$ based at $x_0$
intersecting $e(\bdr U)$ in two antipodal points. Let $\alpha$ and $\beta$ be the images in $N_{g+1}$ of $e(\bdr{U})$ and $\gamma\backslash int(e(U))$ respectively.
If $\gamma$ is one-sided, then 
\[\psi([\gamma])=Y_{\alpha,\beta}.\]
If $\gamma$ is two-sided, then
\[\psi([\gamma])=(T_{\delta_1} T_{\delta_2}^{-1})^\varepsilon,\]
where $\delta_1$ and $\delta_2$ are boundary curves of a regular neighbourhood $M$ of $\alpha\cup\beta$, the twists are right with respect to some orientation of $M\backslash\alpha$ and $\varepsilon$ is $1$ or $-1$ depending on the orientation of $\beta$.
\end{lemma}

Lemma \ref{push} suggests the following generalization of the definition of a crosscap slide. Let $\alpha$ and $\beta$ be one-sided curves intersecting in one point. Then we define  $Y_{\alpha,\beta}$ as
\[Y_{\alpha,\beta}=(T_{\delta_1} T_{\delta_2}^{-1})^\varepsilon,\]
where $\delta_1$ and $\delta_2$ are as in Lemma \ref{push} (see Figure \ref{Y2}).
For such generalized definition the properties (\ref{y^-11}), (\ref{y^-1}) and (\ref{hyh}) remain valid.

\begin{figure}
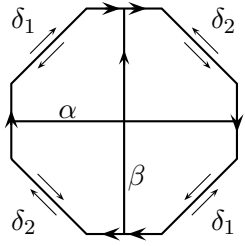

\pspicture*(3.5,3.5)
\rput[tl](.25,3.25){\small$\delta_1$}
\rput[tr](3.25,3.25){\small$\delta_2$}
\rput[bl](.25,.25){\small$\delta_2$}
\rput[br](3.25,.25){\small$\delta_1$}
\rput[b](1,1.8){\small$\alpha$}
\rput[l](1.8,1){\small$\beta$}
\psline[arrowsize=4pt]{->}(1.75,2.5)(1.75,2.7)
\psline(.25,1.25)(.25,2.25)(1.25,3.25)(2.25,3.25)(3.25,2.25)(3.25,1.25)
(2.25,.25)(1.25,.25)(.25,1.25)
\psline[arrowsize=5pt]{->}(.25,1.7)(.25,1.9)
\psline[arrowsize=5pt]{<-}(3.25,1.6)(3.25,1.8)
\psline[arrowsize=5pt]{->}(1.5,3.25)(1.7,3.25)
\psline[arrowsize=5pt]{->}(1.85,3.25)(2.05,3.25)
\psline[arrowsize=5pt]{<-}(1.8,.25)(2,.25)
\psline[arrowsize=5pt]{<-}(1.45,.25)(1.56,.25)
\psline(.25,1.75)(3.25,1.75)
\psline(1.75,.25)(1.75,3.25)
\psline[linewidth=.4pt]{->}(.5,2.65)(.85,3)
\psline[linewidth=.4pt]{<-}(.6,2.45)(.95,2.8)
\psline[linewidth=.4pt]{->}(3,2.65)(2.65,3)
\psline[linewidth=.4pt]{<-}(2.9,2.45)(2.55,2.8)
\psline[linewidth=.4pt]{<-}(.5,.85)(.85,.5)
\psline[linewidth=.4pt]{->}(.6,1.05)(.95,.7)
\psline[linewidth=.4pt]{<-}(3,.85)(2.65,.5)
\psline[linewidth=.4pt]{->}(2.9,1.05)(2.55,.7)
\endpspicture
\caption{\label{Y2}  $Y_{\alpha,\beta}$ for two one-sided curves.}
\end{figure}

\section{Generators of the level 2 mapping class group}
Let us  represent $N_g$ as a 2-sphere with $g$ crosscaps. This means that interiors of $g$ small pairwise disjoint discs should be removed from the sphere, and then antipodal points in each of the resulting boundary components should be identified. Let us arrange the crosscaps as shown on Figure \ref{aI} and number them from $1$ to $g$. 
For each nonempty subset $I\subseteq\{1,\dots,g\}$ let $\alpha_I$ be the simple closed curve shown on Figure \ref{aI}. For $I=\{i_1,\dots,i_k\}$ let $|I|=k$. Note that $\alpha_I$ is two-sided if and only if $|I|$ is even. In such case $T_{\alpha_I}$ will be Dehn twist about $\alpha_I$ in the direction indicated by arrows on Figure \ref{aI}. We will write $\alpha_i$ instead of $\alpha_{\{i\}}$ for $i\in\{1,\dots,g\}$.
\begin{figure}
\input{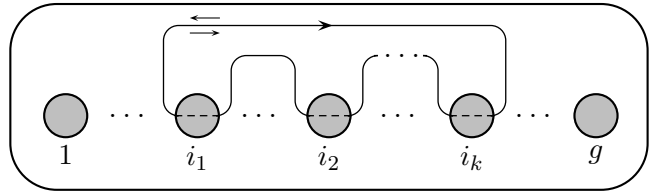}
\caption{\label{aI} The curve $\alpha_I$ for $I=\{i_1,i_2,\dots,i_k\}$.}
\end{figure}

It is well know that $\M(N_1)$ is trivial and $\M(N_2)\approx\Z_2\times\Z_2$ (\cite{Lick1}). It follows easily from the structure of $\M(N_2)$ that $\Gamma_2(N_2)$ has order two and it is generated by a crosscap slide.  
The following theorem can be deduced from Chillingworth's result by reducing the generating set given in \cite{Chill} by using the method of Humphries \cite{Hum1} (cf. \cite[Theorem 3]{Szep3}).
\begin{theorem}\label{chil}
For $g\ge 4$ the mapping class group $\M(N_g)$ is generated by the following elements.
\begin{itemize}
\item $Y=Y_{\alpha_{1},\alpha_{\{1,2\}}}$, 
\item $A_i=T_{\alpha_{\{i,i+1\}}}$ for $i=1,\dots,g-1$,
\item $B=T_{\alpha_{\{1,2,3,4\}}}$.
\end{itemize}
The group $\M(N_3)$ is generated by  $Y$, $A_1$, $A_2$.
\end{theorem}
The following theorem is the main result of this paper.
\begin{theorem}\label{main}
For $g\ge 4$ the level 2 mapping class group $\Gamma_2(N_g)$ is generated by the following elements.
\begin{itemize}
\item[(1)] $Y_{\alpha_i,\alpha_{\{i,j\}}}$ for $i\in\{1,\dots,g-1\}$, $j\in\{1,\dots,g\}$ and $i\ne j$, 
\item[(2)] $Y_{\alpha_{\{i,j,k\}},\alpha_{\{i,j,k,l\}}}$ for $i<j<k<l$.
\end{itemize}
The group $\Gamma_2(N_3)$ is generated by the elements (1).
\end{theorem}
\begin{rem}
There are $(g-1)^2$ generators of type (1) and $\binom{g}{4}$ generators of type (2).
Thus we have $4$ generators for $\Gamma_2(N_3)$ and $10$ generators for $\Gamma_2(N_4)$. We will show in Section \ref{low} that these are minimal numbers of generators for these groups. We do not expect that Theorem \ref{main} provides minimal number of generators for $\Gamma_2(N_g)$ for $g>4$. 
\end{rem}

Let $G$ be the subgroup of $\M(N_g)$ generated by the elements (1), (2) from Theorem \ref{main}. Our goal is to prove that $G=\Gamma_2(N_g)$. First we need to prove some lemmas.

\begin{lemma}\label{t^21}
Suppose that $\alpha$ and $\beta$ are two simple closed curves 
such that $\alpha$ is one-sided, $\beta$ is two-sided and they
intersect in one point. Then 
\[T_\beta^2=Y_{T_\beta(\alpha),\beta}Y^{-1}_{\alpha,\beta}\]
\end{lemma}
\begin{proof}
Since $Y_{\alpha,\beta}$ preserves the curve $\beta$ and reverses orientation of its neighbourhood, we have \[Y_{\alpha,\beta}T_\beta Y^{-1}_{\alpha,\beta}=T^{-1}_{\beta}.\] 
On the other hand, by (\ref{hyh}) we have
\begin{align*}
&T_{\beta}Y_{\alpha,\beta}T^{-1}_{\beta}=Y_{T_\beta(\alpha),\beta},\\
&T_{\beta}Y_{\alpha,\beta}T^{-1}_{\beta}Y^{-1}_{\alpha,\beta}=Y_{T_\beta(\alpha),\beta}Y^{-1}_{\alpha,\beta},\\
&T_\beta^2=Y_{T_\beta(\alpha),\beta}Y^{-1}_{\alpha,\beta}.
\end{align*}
\end{proof}

\begin{lemma}\label{t^22}
Suppose that $\alpha$ and $\beta$ are one-sided simple closed curves intersecting in one point. Let $\delta$ be a boundary curve of a regular neighbourhood of $\alpha\cup\beta$. Then
\[T^2_\delta=Y_{\alpha,\beta}^{\varepsilon_1}Y_{\beta,\alpha}^{\varepsilon_2},\]
where $\varepsilon_i$ is $1$ or $-1$ for $i=1,2$.
\end{lemma}
\begin{proof}
Let $\delta=\delta_1$ and $\delta_2$ be the boundary curves of a regular neighbourhood $M$ of $\alpha\cup\beta$ and suppose that $T_{\delta_1}$ and $T_{\delta_2}$ are right twists with respect to some orientation of $M\backslash\alpha$. Thus  $Y_{\alpha,\beta}=(T_{\delta_1} T_{\delta_2}^{-1})^{\varepsilon_1}.$ Observe that with respect to any orientation of $M\backslash\beta$ one of the twists $T_{\delta_i}$ is right, while the other one is left. Hence  $Y_{\beta,\alpha}=(T_{\delta_1} T_{\delta_2})^{\varepsilon_2}$ and the lemma follows.
 \end{proof}

\begin{lemma}\label{I1}
For $i\in\{1,\dots,g\}$ and for every simple closed curve $\beta$ intersecting $\alpha_i$ in one point we have $Y_{\alpha_i,\beta}\in G$. Moreover, $Y_{\alpha_i,\beta}$ can be written as a product of generators of type (1).  
\end{lemma}
\begin{proof}
Let $G'$ be the subgroup of $G$ generated by the elements (1).
First we assume $i<g$. Let $N'$ be the surface of genus $g-1$ obtained from $N_g$ by replacing the $i$-th crosscap by a disc $U$ with basepoint $x_0$. As $N_g$ may be seen as being obtained from $N'$ by the blowup construction, we have the corresponding crosscap pushing map
\[\psi\colon\pi_1(N',x_0)\to\M(N_{g}).\]    
Note that $Y_{\alpha_i,\beta}$ is in the image of $\psi$ (see Lemma \ref{push}).
The group $\pi_1(N',x_0)$ is generated by homotopy classes of simple loops 
$[\gamma_j]$ such that $\psi([\gamma_j])=Y_{\alpha_i,\alpha_{\{i,j\}}}$ for $j\in\{1,\dots,i-1,i+1,\dots,g\}$ (in fact $[\gamma_j]$ can be taken to be standard generators of the fundamental group). It follows that  $\psi(\pi_1(N',x_0))\subseteq G'$ and $Y_{\alpha_i,\beta}\in G'$.

Now suppose $i=g$. It suffices to show that $Y_{\alpha_g,\alpha_{\{j,g\}}}\in G'$ $j\in\{1,\dots,g-1\}$, the rest of the proof follows as above. Note that $T_{\alpha_{\{j,g\}}}(\alpha_j)=\alpha_g^{-1}$. By Lemma \ref{t^21} we have  \[Y_{\alpha_g,\alpha_{\{j,g\}}}=T_{\alpha_{\{j,g\}}}^2Y_{\alpha_j,\alpha_{\{j,g\}}}\] and hence it suffices to show that $T^2_{\alpha_{\{j,g\}}}\in G'$.
Without loss of generality we assume $j=g-1$. Let $M$ be the surface obtained by cutting $N_g$ along $\alpha_{\{g-1,g\}}$. Thus $M$ is a nonorientable surface of genus $g-2$ with two boundary components. For $n\in\{1,\dots,g-1\}$ there exist pairwise disjoint two-sided simple closed curves $\delta_n$ in $M$ such that $\bdr{M}=\delta_1\cup\delta_{g-1}$ and $\delta_{k}\cup\delta_{k+1}=\bdr{M_k}$, where 
$M_k$ is a genus one subsurface containing $\alpha_k$ for $k\in\{1,\dots,g-2\}$.
Choose an orientation of $M\backslash(\alpha_1\cup\dots\cup\alpha_{g-2}$ and let $T_{\delta_n}$ be the right Dehn twist with respect to that orientation for $n\in\{1,\dots,g-1\}$. For $k\in\{1,\dots,g-2\}$ there exists a one-sided simple curve $\beta_k$ in $M_k$ intersecting $\alpha_k$ in one point and such that
$Y_{\alpha_k,\beta_k}=T_{\delta_k}T^{-1}_{\delta_{k+1}}$. We have
\[Y_{\alpha_1,\beta_1}Y_{\alpha_2,\beta_2}\cdots Y_{\alpha_{g-2},\beta_{g-2}}=T_{\delta_1}T^{-1}_{\delta_{g-1}},\]
and after recovering $N_g$ by gluing together the boundary curves of $M$ we obtain 
\[Y_{\alpha_1,\beta_1}Y_{\alpha_2,\beta_2}\cdots Y_{\alpha_{g-2},\beta_{g-2}}=T^{\pm 2}_{\alpha_{\{g-1,g\}}}.\]
This completes the proof because $Y_{\alpha_k,\beta_k}\in G'$ for
$k\in\{1,\dots,g-2\}$ by earlier part of the proof.
\end{proof}

\begin{lemma}\label{t23}
For every $I\subseteq\{1,\dots,g\}$ such that $|I|=2$ or $|I|=4$ we have $T^2_{\alpha_I}\in G$.
\end{lemma}
\begin{proof}
Suppose that $I=\{i,j\}$ where $i<j$. Observe that $T_{\alpha_{\{i,j\}}}(\alpha_i)=\alpha_j^{-1}$. By Lemma \ref{t^21} we have  $T_{\alpha_{\{i,j\}}}^2=Y_{\alpha_j,\alpha_{\{i,j\}}}Y^{-1}_{\alpha_i,\alpha_{\{i,j\}}}\in G$.

Suppose that $I=\{i,j,k,l\}$ where $i<j<k<l$ and let $J=\{i,j,k\}$. Observe that $T_{\alpha_I}(\alpha_J)=\alpha_l^{-1}$. By Lemma \ref{t^21} we have  $T_{\alpha_I}^2=Y_{\alpha_l,\alpha_I}Y^{-1}_{\alpha_J,\alpha_I}$. Since  $Y_{\alpha_l,\alpha_I}\in G$ by Lemma \ref{I1}, and $Y_{\alpha_J,\alpha_I}\in G$ by the definition of G, also $T_{\alpha_I}^2\in G$.
\end{proof}

\begin{lemma}\label{I3}
For every $I=\{i,j,k\}$, where $1\le i<j<k\le g$, and for every two-sided simple closed curve $\beta$ intersecting $\alpha_I$ in one point we have $Y_{\alpha_I,\beta}\in G$.  
\end{lemma}
\begin{proof}
Note that $\beta$ actually exists only for $g\ge 4$. Indeed, if $g=3$ then the surface obtained from $N_g$ by cutting along $\alpha_I$ is orientable, and every curve intersecting $\alpha_I$ in one point must be one-sided. Therefore we are assuming $g\ge 4$. 

 Let $H$ be the subgroup of $\M(N)$ generated by $Y_{\alpha_I,\alpha_i}$,
$Y_{\alpha_I,\alpha_k}$ and $Y_{\alpha_I,\alpha_J}$ for every $J$ such that $|J|=4$ and $I\subset J$. We claim that $H\subseteq G$. Consider a regular neighborhood of $\alpha_I\cup\alpha_i$. One of its boundary curves is $\alpha_{\{j,k\}}$ and by Lemma \ref{t^22} we have
\[T^2_{\alpha_{\{j,k\}}}=Y_{\alpha_I,\alpha_i}^{\varepsilon_1}Y_{\alpha_i,\alpha_I}^{\varepsilon_2},\]
where $\varepsilon_l$ is $1$ or $-1$ for $l=1,2$. Since $Y_{\alpha_i,\alpha_I}\in G$ by Lemma \ref{I1} and $T^2_{\alpha_{\{j,k\}}}\in G$ by Lemma \ref{t23}, also $Y_{\alpha_I,\alpha_i}\in G$. By a similar argument  $Y_{\alpha_I,\alpha_k}\in G$.
Let $J=I\cup\{l\}$ for $l\notin\{i,j,k\}$ and note that 
$T^{\pm 1}_{\alpha_J}(\alpha_I)=\alpha_l^{-1}$. By Lemma \ref{t^21} we have 
\[T_{\alpha_J}^{\pm 2}=Y_{\alpha_l,\alpha_J}Y^{-1}_{\alpha_I,\alpha_J},\]
and since $Y_{\alpha_l,\alpha_J}\in G$ by Lemma \ref{I1} and $T^2_{\alpha_J}\in G$ by Lemma \ref{t23}, also $Y_{\alpha_I,\alpha_J}\in G$ and the claim is proved.

Now it suffices to show that $Y_{\alpha_I,\beta}\in H$. This can be done as follows. Choose $n\in \{1,\dots,g\}\backslash\{i,j,k\}$ and observe that for
$J=I\cup\{n\}$ we have $T^\varepsilon_{\alpha_J}(\alpha_I)=\alpha_n^{-1}$, where $\varepsilon=\pm 1$. Let
\[Y=T^\varepsilon_{\alpha_J}Y_{\alpha_I,\beta}T^{-\varepsilon}_{\alpha_J}
=Y_{\alpha_n,T^\varepsilon_{\alpha_J}(\beta)},\qquad
H'=T^\varepsilon_{\alpha_J}HT^{-\varepsilon}_{\alpha_J}.\] 
We need to show that $Y\in H'$. 
Let $N'$ be the surface of genus $g-1$ obtained from $N_g$ by replacing the $n$-th crosscap by a disc $U$ with basepoint $x_0$. We have the crosscap pushing map
\[\psi\colon\pi_1(N',x_0)\to\M(N_{g}).\]     
Since $Y\in \psi(\pi_1(N',x_0))$ it suffices to show that
$\psi(\pi_1(N',x_0))\subseteq H'$ and for that it is enough to check that 
$\pi_1(N',x_0)$ is generated by loops mapped by $\psi$ on generators of $H'$. Let us assume $n=1$. The proof is similar for other $n$.
In this case we have $H'=D^{-1}HD$ for $D=T_{\alpha_{\{1,i,j,k\}}}$. 
For
$s\in\{2,\dots,g\}$ let $x_s$ be the standard generators of $\pi_1(N',x_0)$
 such that $\psi(x_s)=Y_{\alpha_1,\alpha_{\{1,s\}}}$. The group $H'$ is generated by
\begin{align*}
&D^{-1}Y_{\alpha_I,\alpha_i}D=Y_{\alpha_1,\alpha^{-1}_{\{1,j,k\}}}=\psi((x_jx_k)^{-1}),\\
&D^{-1}Y_{\alpha_I,\alpha_k}D=Y_{\alpha_1,\alpha^{-1}_{\{1,i,j\}}}=\psi((x_ix_j)^{-1}),\\
&D^{-1}Y_{\alpha_I,\alpha_{\{1,i,j,k\}}}D=Y_{\alpha_1,\alpha_{\{1,i,j,k\}}}=\psi(x_ix_jx_k),\\
&D^{-1}Y_{\alpha_I,\alpha_{\{l,i,j,k\}}}D=\psi((x_ix_jx_k)^{-1}x_lx_ix_jx_k)\quad\textrm{for\ $1<l<i$},\\
&D^{-1}Y_{\alpha_I,\alpha_{\{i,l,j,k\}}}D=\psi(x_ix_lx_i^{-1})\quad\textrm{for\ $i<l<j$},\\
&D^{-1}Y_{\alpha_I,\alpha_{\{i,j,l,k\}}}D=\psi(x_k^{-1}x_lx_k)\quad\textrm{for\ $j<l<k$},\\
&D^{-1}Y_{\alpha_I,\alpha_{\{i,j,k,l\}}}D=\psi(x_ix_jx_kx_l(x_ix_jx_k)^{-1})\quad\textrm{for\ $k<l\le g$}.
\end{align*}
It is easy to check that each $x_s$ can be expressed as a product of elements which are mapped by $\psi$ on the generators of $H'$. It follows that
$\psi(\pi_1(N',x_0))\subseteq H'$ and the lemma is proved.
\end{proof}

\begin{lemma}\label{I5}
Let $I=\{i_1,\dots,i_5\}$ and $J=\{i_1,\dots,i_5,i_6\}$, where
$1\le i_1<\cdots<i_5<i_6\le g$. Then $Y_{\alpha_I,\alpha_J}\in G$.
\end{lemma}
\begin{proof}
Note that $T_{\alpha_J}(\alpha_{\{i_1,i_2,i_3\}})=\alpha^{-1}_{\{i_4,i_5,i_6\}}$. By Lemma \ref{t^21} we have
\[T_{\alpha_J}^2=Y_{\alpha_{\{i_4,i_5,i_6\}},\alpha_J}Y^{-1}_{\alpha_{\{i_1,i_2,i_3\}},{\alpha_J}}.\]
On the other hand, since $T_{\alpha_J}(\alpha_I)=\alpha^{-1}_{i_6}$, we also have
\[T_{\alpha_J}^2=Y_{\alpha_{i_6},\alpha_J}Y^{-1}_{\alpha_I,{\alpha_J}}.\]
Hence
\[Y_{\alpha_I,{\alpha_J}}=Y_{\alpha_{\{i_1,i_2,i_3\}},\alpha_J}Y^{-1}_{\alpha_{\{i_4,i_5,i_6\}},{\alpha_J}}Y_{\alpha_{i_6},\alpha_J}\in G\]
by Lemmas \ref{I1} and \ref{I3}.
\end{proof}

\noindent{\it Proof of Theorem \ref{main}.}
Since by \cite[Lemma 3.6]{Szep2} $\Gamma_2(N_g)$ is generated by crosscap slides conjugate to $Y_{\alpha_1,\alpha_{\{1,2\}}}$ it suffices to prove that $G$ is normal in $\M(N_g)$. By Theorem \ref{chil} it is enough to check for every generator $x$ of $G$ that $A_ixA_i^{-1}\in G$ or $A_i^{-1}xA_i\in G$ for $i\in\{1,\dots,g\}$ and $BxB^{-1}\in G$ or $B^{-1}xB\in G$. Note that
$A_ixA_i^{-1}\in G\iff A_i^{-1}xA_i\in G$ since $A_i^2\in G$ by Lemma \ref{t23}, and analogously for $B$.

For $i\in\{1,\dots,g\}$ we have $A^{-1}_{i-1}(\alpha_i)=\alpha_{i-1}^{-1}$, $A_i(\alpha_i)=\alpha_{i+1}^{-1}$ and $A_k(\alpha_i)=\alpha_i$ for $k\neq i-1,i$. It follows that for all $i\ne j$ and $k$ the generator $Y_{\alpha_i,\alpha_{\{i,j\}}}$ is conjugated by $A_k$ or $A_k^{-1}$ to
$Y_{\alpha_l,\beta}$ for some $l$ and some $\beta$. 
Since the last element is in $G$ by Lemma \ref{I1}, we have proved that   $A_ixA_i^{-1}\in G$ for every $i\in\{1,\dots,g\}$ and every generator $x$ of type (1).

If $g\ge 4$ then $B(\alpha_i)=\alpha_i$ for $i>4$, while for $i\in\{1,2,3,4\}$ we have $B^{\pm 1}(\alpha_i)=\alpha_I^{-1}$, where $I=\{1,2,3,4\}\backslash\{i\}$.  
It follows by Lemma \ref{I1} or by Lemma \ref{I3} that $BxB^{-1}\in G$ for every generator $x$ of type (1).

Suppose that $I=\{i,j,k\}$ and $I=\{i,j,k,l\}$, where $i<j<k<l$. It can be checked that for $i>1$ we have 
$A_{i-1}(\alpha_I)=Y_{\alpha_i,\alpha_{\{i-1,i\}}}(\alpha_{\{i-1,j,k\}})$ and
$A_{i-1}(\alpha_J)=Y_{\alpha_i,\alpha_{\{i-1,i\}}}(\alpha_{\{i-1,j,k,l\}})$.
It follows that \[A_{i-1}Y_{\alpha_I,\alpha_J}A_{i-1}^{-1}
=Y_{\alpha_i,\alpha_{\{i-1,i\}}}Y_{\alpha_{\{i-1,j,k\}},\alpha_{\{i-1,j,k,l\}}}Y^{-1}_{\alpha_i,\alpha_{\{i-1,i\}}}\in G.\]
By similar arguments one can check that $A_nY_{\alpha_I,\alpha_J}A_n^{-1}\in G$ for $n\in\{1,\dots,g-1\}$.

If $i>4$ then 
$B(\alpha_I)=\alpha_I$, while for $k\le 4$ we have $B^{\pm 1}(\alpha_I)=\alpha_m^{-1}$, where $\{m\}=\{1,2,3,4\}\backslash I$. In both cases we have $BY_{\alpha_I,\alpha_J}B^{-1}\in G$.
If $i=1$, $j=2$ and $k>4$, then for
$Y_1=Y_{\alpha_3,\alpha_{\{2,3\}}}$, $Y_2=Y_{\alpha_4,\alpha_{\{2,4\}}}$
it can be checked  that 
$Y_1Y_2(\alpha_I)$ and $Y_1Y_2(\alpha_J)$ are disjoint from
$\alpha_{\{1,2,3,4\}}$. 
It follows that
\[BY_1Y_2Y_{\alpha_I,\alpha_J}Y^{-1}_2Y^{-1}_1B^{-1}
=Y_1Y_2Y_{\alpha_I,\alpha_J}Y^{-1}_2Y^{-1}_1\in G.\]
From earlier part of the proof we know that $BY_1Y_2B^{-1}\in G$, hence $BY_{\alpha_I,\alpha_J}B^{-1}\in G$. A similar argument, using different $Y_1$, $Y_2$, can be applied to other $I$ such that $|I\cap\{1,2,3,4\}|=2$. 
It remains to consider the cases where $i\le 4$ and $j>4$.
If $i=1$ then it cen be checked that
$B^{-1}(\alpha_I)=Y^{-1}_{\alpha_1,\alpha_{\{1,2\}}}(\alpha_{I'})$, where $I'=\{2,3,4,j,k\}$, and 
$B^{-1}(\alpha_J)=Y^{-1}_{\alpha_1,\alpha_{\{1,2\}}}(\alpha_{J'})$, where  $J'=I'\cup\{l\}$. Since $Y_{\alpha_{I'},\alpha_{J'}}\in G$ by Lemma \ref{I5}, we have
\[B^{-1}Y_{\alpha_I,\alpha_J}B=Y^{-1}_{\alpha_1,\alpha_{\{1,2\}}}Y_{\alpha_{I'},\alpha_{J'}}Y_{\alpha_1,\alpha_{\{1,2\}}}\in G.\]
If $i=2$ then for $Y_1=Y_{\alpha_1,\alpha_{\{1,2\}}}$ and $Y_2=Y_{\alpha_2,\alpha_{\{1,2\}}}$ we have
$B^{-1}Y^{-1}_1(\alpha_I)=Y_1^{-1}Y_2(\alpha_{I'})$, where $I'=\{1,3,4,j,k\}$,
and $B^{-1}Y^{-1}_1(\alpha_J)=Y_1^{-1}Y_2(\alpha_{J'})$,
where $J'=I'\cup\{l\}$.
Since $Y_{\alpha_{I'},\alpha_{J'}}\in G$ by Lemma \ref{I5} we have
\[B^{-1}Y_1^{-1}Y_{\alpha_I,\alpha_J}Y_1B=Y_1^{-1}Y_2Y_{\alpha_{I'},\alpha_{J'}}Y_2^{-1}Y_1\in G,\]
and since $B^{-1}Y_1B\in G$ by earlier part of the proof, also $B^{-1}Y_{\alpha_I,\alpha_J}B\in G$.
The proof is similar for $i=3$ and $i=4$. 
\begin{rem}
By the proof of Lemma \ref{t23} for $i<j<k<l$ we have
\[T^2_{\alpha_{\{i,j,k,l\}}}=Y_{\alpha_l,\alpha_{\{i,j,k,l\}}}Y^{-1}_{\alpha_{\{i,j,k\}},\alpha_{\{i,j,k,l\}}},\]
and by Lemma \ref{I1} $Y_{\alpha_l,\alpha_{\{i,j,k,l\}}}$ can be written as a product of the generators of type (1). It follows that each generator of type (2) $Y_{\alpha_{\{i,j,k\}},\alpha_{\{i,j,k,l\}}}$ can be replaced by
$T^2_{\alpha_{\{i,j,k,l\}}}$ in Theorem \ref{main}.
\end{rem}
\section{Low genus cases}\label{low}
For $i\in\{1,\dots,g\}$ let $c_i$ denote the homology class of the curve $\alpha_i$ in $H_1(N_g,\Z)$. Then $H_1(N_g,\Z)$ has the following presentation as a $\Z$-module:
\[H_1(N_g,\Z)=\lr{c_1,\dots,c_g\,|\,2(c_1+\cdots+c_g)=0}.\]
Consider the quotient $R_g=H_1(N_g,\Z)/\lr{c}$, where $c=c_1+\cdots+c_g$ is the unique homology class of order 2. It is immediate from the above presentation  that $R_g$ is the free $\Z$-module with basis given by the images of $c_1,\dots,c_{g-1}$ in $R_g$. Let us fix this basis and identify 
$Aut(R_g)$ with $GL(g-1,\Z)$.
Every automorphism of $H_1(N_g,\Z)$ preserves $c$, and thus induces an automorphism of $R_g$. Thus the action of $\M(N_g)$ on
$H_1(N_g,\Z)$ induces a homomorphism
\[\rho\colon\M(N_g)\to GL(g-1,\Z).\]
In general $\rho$ is neither surjective nor injective. 
However, it was shown in \cite[Section 2]{McCP}, that the group of automorphisms of $H_1(N_g,\Z)$ which act trivially on $H_1(N_g,\Z_2)$ is isomorphic to the full group of automorphisms of $R_g$ which act trivially on $R_g\otimes\Z_2$. Consequently, the restriction of $\rho$ to $\Gamma_2(N_g)$ yields a surjection
$$\eta\colon\Gamma_2(N_g)\longrightarrow GL_2(g-1,\Z),$$
where $GL_2(n,\Z)$ is the level 2 congruence subgroup of $GL(n,\Z)$.

Birman and Chillingworth obtained in \cite[Theorem 3]{BirChill} a finite presentation for $\M(N_3)$ from which it is immediate that this group is isomorphic to $GL(2,\Z)$. 
It turns out that such isomorphism can also be deduced from the action on $H_1(N_3,\Z)$.
\begin{prop}\label{gl2z}
The map $\rho\colon\M(N_3)\to GL(2,\Z)$ is an isomorphism.
\end{prop}
\begin{proof}
It is a classical result that $SL(2,\Z)$ admits the presentation
\[SL(2,\Z)=\lr{S,T\,|\,STS=TST, (STS)^4=I},\]
where $S=\begin{pmatrix}1&1\\0&1\end{pmatrix}$,  
$T=\begin{pmatrix}1&0\\-1&1\end{pmatrix}$ and $I$ is the identity matrix. In order to obtain a presentation for $GL(2,\Z)$ one has to add a generator with determinant $-1$, for example
$U=\begin{pmatrix}-1&0\\0&1\end{pmatrix}$ and relations
\[U^2=(US)^2=(UT)^2=I.\]
Let $A_1$, $A_2$ and $Y$ be the generators of $\M(N_3)$ from Theorem \ref{chil}.
It can be checked that 
\[\rho(A_1)=STS^{-1},\  
\rho(A_2)=S,\ 
\rho(Y)=SUS^{-1}.  
\]
In $\M(N_3)$ we have the relations (see \cite{BirChill}):
\begin{align*}
&A_1A_2A_1=A_2A_1A_2,\quad (A_1A_2A_1)^4=1,\\
&Y^2=(YA_1)^2=(YA_2)^2=1.
\end{align*}
It follows that the mapping
\[S\mapsto A_2,\quad T\mapsto A_2^{-1}A_1A_2,\quad U\mapsto A_2^{-1}YA_2\]
extends to a homomorphism $GL(2,\Z)\to\M(N_3)$ which is the inverse of $\rho$.
\end{proof}
\begin{cor}
The map $\eta\colon\Gamma_2(N_3)\to GL_2(2,\Z)$ is an isomorphism.
\end{cor}
Let $Mat_n(\Z_2)$ denote the additive group of $n\times n$ matrices with entries in $\Z_2$. This is an abelian group isomorphic to $\Z_2^{n^2}$. Let us define an epimorphism $f\colon GL_2(n,\Z)\to Mat_n(\Z_2)$. Let $X$ be any matrix in $GL_2(n,\Z)$. Write $X=I+2A$, where $I$ is the identity matrix and define
$f(X)=A\bmod 2$. To see that this is a homomorphism take $Y=I+2B$. Then
\[f(XY)=f(I+2(A+B)+4AB)=A+B \bmod 2.\]
Let $E_{i,j}$ be the elementary $n\times n$ matrix with $1$ at position $(i,j)$ and $0$'s elsewhere. Since for each pair $(i,j)$ the matrix $I-2E_{i,j}$ is in $GL_2(n,\Z)$, thus $f$ is onto. The map $f$ was defined in \cite{LSz} to determine  abelianizations of congruence subgroups of $SL(n,\Z)$. Now let $g>1$ and consider the composition
\[f\circ\eta\colon \Gamma_2(N_g)\to Mat_{g-1}(\Z_2).\]
Since $f\circ\eta$ is surjective, we see that $\Gamma_2(N_g)$ cannot be generated by less then $(g-1)^2$ elements. In particular Theorem \ref{main} provides the minimal number of generators for $\Gamma_2(N_3)$. It follows from the next proposition that this is also the case for $g=4$.
\begin{prop}
The group $H_1(\Gamma_2(N_4),\Z)$ is isomorphic to $\Z_2^{10}$.
\end{prop}
\begin{proof}
By \cite[Theorem 3.7]{Szep2} $\Gamma_2(N_g)$ is generated by involutions. It follows that $H_1(\Gamma_2(N_4),\Z)\approx\Z_2^d$ for some integer $d$. From Theorem \ref{main} we have $d\le 10$ and the existence of $f\circ\eta$ gives $d\ge 9$. Let
\[h\colon H_1(\Gamma_2(N_4),\Z)\to Mat_3(\Z_2)\] be the map induced by $f\circ\eta$. To prove $d=10$ it suffices to show that $\ker h$ is not trivial.
Let $B=T_{\alpha_{\{1,2,3,4\}}}$ and observe that $B^2$ induces the identity on $H_1(N_4,\Z)$, hence it belongs to $\ker\eta$. We claim that 
$B^2$ is not in the commutator subgroup $[\Gamma_2(N_4),\Gamma_2(N_4)]$, hence it represents a nontrivial element of $\ker h$. To prove this claim we need to refer to the presentation of $\M(N_4)$ given in \cite{Szep1}. It follows from this presentation that there exists an epimorphism
$\theta\colon\M(N_4)\to\Z_2$ such that $\theta(B)=1$ and $\theta(x)=0$ for every generator $x$ different from $B$. In particular $\theta(Y)=0$, where $Y=Y_{\alpha_1,\alpha_{\{1,2\}}}$. Since $\Gamma_2(N_4)$ is the normal closure of $Y$ we have $\Gamma_2(N_4)\subset\ker\theta$. It is a routine to obtain from the presentation of $\M(N_4)$ a presentation for the index 2 subgroup $\ker\theta$ and check that $B^2$ survives in its abalinization, that is $B^2\notin[\ker\theta,\ker\theta]$. Since $\Gamma_2(N_4)\subset\ker\theta$ also
$B^2\notin [\Gamma_2(N_4),\Gamma_2(N_4)]$.
\end{proof}
%

%
%

\end{document}